\documentclass[letterpaper, 10 pt, conference]{ieeeconf}
\usepackage{mymacros}

\IEEEoverridecommandlockouts
\usepackage[left=1.9cm,right=1.9cm,top=1.9cm,bottom=1.9cm]{geometry}

\IEEEoverridecommandlockouts
\title{\LARGE \bf
Automated Performance Estimation for Decentralized Optimization via Network Size Independent Problems
}

\author{Sebastien Colla and Julien M. Hendrickx
\thanks{S. Colla  and  J.  M.  Hendrickx  are  with  the  ICTEAM institute, UCLouvain, Louvain-la-Neuve, Belgium. S. Colla is supported by the French Community of Belgium through a FRIA fellowship (F.R.S.-FNRS). J. M. Hendrickx is supported by the “RevealFlight” Concerted Research Action (ARC) of the Federation Wallonie-Bruxelles and by the F.R.S.-FNRS via its research project Kornet and its Incentive Grant for Scientific Research (MIS) “Learning from Pairwise Comparisons”. Email addresses: {\tt\small sebastien.colla@uclouvain.be}, {\tt\small julien.hendrickx@uclouvain.be}}
}

\begin{document}

\maketitle
\thispagestyle{empty}
\pagestyle{empty}

%%%%%%%%%%%%%%%%%%%%%%%%%%%%%%%%%%%%%%%%%%%%%%%%%%%%%%%%%%%%%%%%%%%%%%%%%%%%%%%%
\begin{abstract}
  We develop a novel formulation of the Performance Estimation Problem (PEP) for decentralized optimization whose size is independent of the number of agents in the network. The PEP approach allows computing automatically the worst-case performance and worst-case instance of first-order optimization methods by solving an SDP. Unlike previous work, the size of our new PEP formulation is independent of the network size. For this purpose, we take a global view of the decentralized problem and we also decouple the consensus subspace and its orthogonal complement. We apply our methodology to different decentralized methods such as DGD, DIGing and EXTRA and obtain numerically tight performance guarantees that are valid for any network size.
\end{abstract}

\section{Introduction}
Decentralized optimization has received an increasing attention due to its useful applications in large-scale machine learning and sensor networks, see e.g. \cite{DGD} for a survey. In decentralized methods for separable objective functions, we consider a set of agents $\{1,\dots,N\}$, working together to solve the following optimization problem: \vspace{-1mm}
\begin{equation} \label{opt:dec_prob}
  \underset{\text{\normalsize $x \in \Rvec{d}$}}{\mathrm{min}} \quad f(x) = \frac{1}{N}\sum_{i=1}^N f_i(x), \vspace{-1mm}
\end{equation}
where $f_i: \Rvec{d}\to\mathbb{R}$ is the private function locally held by agent $i$.
To achieve this goal, each agent $i$ holds its own version $x_i$ of the decision variable $x \in \Rvec{d}$. Agents perform local computations and exchange local information with their neighbors to come to an agreement on the minimizer $x^*$ of the global function $f$.
The problem \eqref{opt:dec_prob} can thus be written using the separable function $F_s:\Rvec{Nd}\to\Rvec{}$: \vspace{-1mm}
\begin{align} \label{opt:dec_prob2}
  \underset{\text{\small $x_1\dots,x_N \in \Rvec{d}$}}{\mathrm{min}} &\quad F_s(x_1,\dots,x_N) = \frac{1}{N}\sum_{i=1}^N f_i(x_i), \hspace{5mm}\\[-0.2mm]
  & \text{such that } \quad x_1 = \dots = x_N. \\[-6mm]
\end{align}
Exchanges of information often take the form of an average consensus step on some quantity, e.g.,  on the $x_i$. This corresponds to a multiplication by an averaging matrix $W \in \Rmat{N}{N}$, often assumed symmetric and \emph{doubly stochastic}, i.e., a nonnegative matrix whose rows and columns sum to one. This matrix $W$ indicates both the topology of the network of agents and the weights they are using during the average consensus. Therefore we call it network or averaging matrix, without distinction.

One of the simplest decentralized optimization method is the distributed (sub)gradient descent (DGD) \cite{DsubGD} where agents successively perform an average consensus step \eqref{eq:DGD_cons} and a local gradient step \eqref{eq:DGD_comp}: \vspace{-1mm}
\begin{align}
  y_i^k &= \sum_{j=1}^N w_{ij} x_j^k,           \hspace*{15mm} \label{eq:DGD_cons} \\[-0.5mm]
  x_i^{k+1} &= y_i^k - \alpha^k \nabla f_i(x_i^k), \hspace*{15mm} \label{eq:DGD_comp} \vspace{-0.5mm}
\end{align}
for step-sizes $\alpha^k > 0$.
We focus on decentralized methods where the only agent-specific part in their update is the consensus step, in which the agents use different averaging weights. This is the case of any method that combines consensus steps, e.g. \eqref{eq:DGD_cons}, and calls to first-order oracles of local functions for implicit or explicit updates, e.g. \eqref{eq:DGD_comp}, using identical step-sizes and parameters for all the agents.
This class of decentralized methods, denoted $\M$, includes many algorithms such as DGD \cite{DsubGD}, EXTRA \cite{EXTRA}, DIGing \cite{DIGing}, and NIDS \cite{NIDS}, to name a few.

In general, the quality of an optimization method is evaluated \textit{via} a worst-case guarantee. Having accurate performance bounds for a decentralized algorithm is important to correctly understand the impact of its parameters and the network topology on its performance. However, obtaining these guarantees can often be a challenging task, requiring combining the impact of the optimization component and of the interconnection network, sometimes resulting in bounds that are conservative or very complex.

\subsection{Contributions and Previous works}
In this work, we build on a computational approach that computes worst-case performance guarantee of an algorithm by solving an optimization problem, known as the performance estimation problem (PEP), see Section \ref{sec:PEPcen} for more details.
The PEP approach has led to many results in centralized optimization, see e.g. \cite{PEP_Smooth,PEP_composite,Taylor_thesis}, and we have recently proposed a formulation tailored for decentralized optimization \cite{PEP_dec_CDC, PEP_dec} using a spectral description for symmetric and \emph{generalized}\footnote{A generalized doubly stochastic matrix has rows and columns that sum to one, without the need to be non-negative, see \cite[Definition 2]{PEP_dec}.} doubly stochastic averaging matrices, leading to some interesting results and observations. This formulation considers each agent individually in the PEP, implying that its size grows with the number of agents $N$. However, we observed in all our experiments \cite[Section IV]{PEP_dec} that the results are independent of $N$ for classical performance criteria.
In this paper, we develop a new PEP formulation for decentralized optimization whose problem size is independent of $N$, though $N$ could still appear as a scaling coefficient in some cases,  see Section \ref{sec:new_form_indep_N}.
This PEP formulation takes a global view on the decentralized problem \eqref{opt:dec_prob2}, by neglecting the separability of function $F_s$, which a priori corresponds to a relaxation. Then,
we apply a change of variables to distinguish the consensus and orthogonal parts of a vector.
In Section \ref{sec:numres}, we demonstrate on several examples (DGD, DIGing and EXTRA), and observe empirically the tightness of this new agent-independent PEP formulation with symmetric and generalized doubly stochastic averaging matrices.
We also observe that EXTRA seems to perform better than DIGing for time-constant averaging matrices but is not robust to time-varying matrices.

\subsection{Related work}
An alternative approach with similar motivations for automated performance evaluation is proposed in \cite{IQC} and is inspired by dynamical systems concepts. Integral quadratic constraints (IQC), usually used to obtain stability guarantees on complex dynamical systems, are adapted to obtain sufficient conditions for the convergence of optimization algorithms
and deduce numerical bounds on the convergence rates.
It has been applied to decentralized optimization in \cite{IQC_dec}.
This methodology analyzes the convergence rate of a single iterate, which is beneficial for the problem size that remains small (and that is also independent of the number of agents in the decentralized case). However, this does not allow to deal with non-geometric convergences, e.g. on smooth convex functions, nor to analyze cases with properties that apply over several iterations, e.g. when averaging matrices are constant.

These are possible with the PEP approach which computes the worst-case performance on a given number $K$ of iterations, but solving problems whose size grows with $K$.

\section{The general PEP approach} \label{sec:PEPcen}
In principle, a tight performance bound on an algorithm could be obtained by running it on every single instance - function and initial condition - allowed by the setting considered and selecting the worst performance obtained. This would also directly provide an example of “worst” instance if it exists.
The performance estimation problem (PEP) formulates this abstract idea as a real optimization problem that maximizes the error measure of the algorithm result, over all possible functions and initial conditions allowed \cite{PEP_Drori}.
This optimization problem is inherently infinite-dimensional, as it contains a continuous function among its variables. Nevertheless, Taylor et al. have shown \cite{PEP_Smooth,PEP_composite} that PEP can be solved exactly using an SDP formulation, for a wide class of centralized first-order algorithms and different classes of functions.

The main ingredients of a PEP are: (i) a performance measure $\mc{P}$, e.g. $f(x^k)-f(x^*)$; (ii) a class of functions $\mc{F}$, e.g. the class $\mc{F}_{\mu,L}$ of $\mu$-strongly convex and $L$-smooth functions; (iii) an optimization method $\mc{M}$; (iv) a set of optimal solutions $\mc{X}^*$; (v) a set $\mc{I}^0$ of initial conditions, e.g. $\|x^0-x^*\|^2 \le 1$.
These ingredients are organized into an optimization problem as \vspace{-1mm}
\begin{align}
  \underset{f, x^0,\dots, x^K, x^*}{\sup} \quad & \mathrm{\mc{P}}\qty(f, x^0,\dots,x^K, x^*) \hspace*{-4mm}  \label{eq:gen_PEP}\\[-0.3mm]
 \text{s.t.} \hspace{8mm}       & \hspace{-5mm}    f \in \mc{F}, \hspace{8mm}
                                 x^* \in \mc{X}^*, \\[-0.6mm]
                                & \hspace{-5mm} x^k \text{ are iterates from method $\mc{M}$,} \\[-0.6mm]
                                & \hspace{-5mm} \mc{I}^0 ~\text{ holds.} \\[-6mm]
\end{align}
To overcome the infinite dimension of variable $f$, the problem can be discretized into $\{\qty(x^k,g^k,f^k)\}_{k\in I}$, where $g^k$ and $f^k$ are respectively the (sub)gradient and the function value of $f$ at point $k$ and $I = \{0,\dots,K,*\}$. Then, the constraint $f \in \mc{F}$ is replaced by interpolation conditions ensuring that there is a function of class $\mc{F}$ which interpolates those data points $\{\qty(x^k,g^k,f^k)\}_{k \in I}$.
Such constraints are provided for many different classes of functions in \cite[Section 3]{PEP_composite}. They are generally quadratics and potentially non-convex in the iterates and the gradients vectors, but they are linear in the scalar products of these and in the function values. The same holds true for most classical performance criteria and initial conditions.
We can then consider these scalar products directly as decision variables of the PEP. For this purpose, we define a Gram matrix $G$ that contains scalar products between iterates $x^k \in \Rvec{d}$ and subgradients $g^k \in \Rvec{d}$. \vspace{-0.25mm}
\begin{align}
  G = P^TP, \text{ with } P = \qty[g^0\dots g^K g^* x^0\dots x^K x^*]. \\[-5.5mm]
\end{align}
By definition, $G$ is symmetric and positive semidefinite. Therefore, under some conditions described below, a PEP in the form of \eqref{eq:gen_PEP} can be formulated as an equivalent positive semidefinite program (SDP) using the Gram matrix $G$ and the vector of functional values $f_v = [f_i]_{i \in I}$ as variables.
This SDP formulation is convenient because it can be solved numerically to global optimality and it provides the worst-case solution over all possible problem dimensions\footnote{Indeed the dimension $d$ of vectors $x^k$ and $g^k$ disappears when taking their scalar products}.
See \cite{PEP_composite} for details about the SDP formulation of PEP, including ways of reducing the size of matrix $G$. However, the dimension of $G$ always depends on the number of iterations $K$.
From a solution $G$, $f_v$ of the SDP formulation, we can construct a solution for the discretized variables $\{\qty(x^k,g^k,f^k)\}_{k\in I}$, e.g. using the Cholesky decomposition. Since these points satisfy sufficient interpolation constraints, we can also construct a function from $\mc{F}$ interpolating these points.

Proposition \ref{prop:GramPEP} states sufficient conditions under which a PEP in the form of \eqref{eq:gen_PEP} can be formulated as an SDP. The proposition uses the following definition.
\begin{definition}[Gram-representable] \label{def:Gram}
Consider a Gram matrix $G$ and a vector $f_v$, as defined above. We say that
a constraint or an objective is linearly (resp. LMI) Gram-representable if it can be expressed using a finite set of linear (resp. LMI) constraints involving (part of) $G$ and $f_v$.
\end{definition} \smallskip
\begin{proposition}[\hspace{-0.5pt}{{\cite[Proposition 2.6]{PEP_composite}}}] \label{prop:GramPEP}
  If the interpolation constraints of the class of functions $\mc{F}$, the satisfaction of the method $\mc{M}$, the performance measure $\mc{P}$ and the set of constraints $\mc{I}$, which includes the initial conditions, are linearly (or LMI) Gram-representable, then, computing the worst-case for criterion $\mc{P}$ of method $\mc{M}$ after $K$ iterations on objective functions in class $\mc{F}$ with constraints $\mc{I}$ can be formulated as an SDP, with $G \succeq 0$ and $f_v$ as variables.
\end{proposition} \smallskip
\emph{Remark:}
In \cite{PEP_composite}, Proposition \ref{prop:GramPEP} was only formulated for linearly Gram-representable constraints, but its extension to LMI Gram-representable constraints is direct. Such constraints appear in the analysis of consensus steps with spectral classes of network matrices.

PEP techniques allowed answering several important questions in optimization, see e.g. the list in \cite{Taylor_thesis}, and to make important progress in the tuning of certain algorithms including the well-known centralized gradient-descent.
It was further exploited to design optimal first-order methods: OGM for smooth convex optimization \cite{OGM} and its extension ITEM for smooth strongly convex optimization \cite{ITEM}.
It can also be used to deduce proofs about the performance of the algorithms \cite{PEP_compo}. It has been made widely accessible \textit{via} a Matlab \cite{PESTO} and Python \cite{pepit2022} toolbox.

\section{Agent-independent PEP formulation for decentralized optimization} \label{sec:new_form_indep_N}

In previous work \cite{PEP_dec}, we find worst-case performance of decentralized optimization algorithms by formulating a performance estimation problem (PEP) that considers individually each agent, with its own local function and local variables.
Remembering that the size of the SDP PEP depends on the number of variables considered and not on their dimension, we will obtain a agent-independent representation by abstracting the variables and functions at the global level. For this, we consider a generalization of problem \eqref{opt:dec_prob2} in which each $f_i$ is allowed to depend on all $x_j$. The problem becomes then one of optimizing a function $F:\Rvec{Nd}\to\R$ over $\x \in \Rvec{Nd}$, under the constraint that $\x$ belongs to the consensus subspace (all $x_i \in \Rvec{d}$ are equal). This will admit a simple representation that only requires decomposing $\x$ into two sets of components, respectively along the consensus subspace and along its orthogonal complement, as opposed to the $N$ sets of components along each $x_i$ in \cite{PEP_dec}.

\subsection{Class of algorithms and vector notations} \label{sec:algo}
We consider the class $\M$ of all the decentralized methods whose update may combine the 3 following types of operations:

(i) \emph{Gradient:}
 Each agent evaluates the (sub)gradient of its local function \vspace{-2.5mm}
 $$ g_i = \nabla f_i(x_i) \qquad \text{for $i=1,\dots,N$}. \vspace{-1.5mm}$$
 Each agent can possibly hold different local functions, which can be involved in different gradient operations.

(ii) \emph{Consensus:}
  All the agents perform a consensus step on any of their local variables \vspace{-1.5mm}
  \begin{equation}
    y_i = \sum_{j=1}^N w_{ij} x_j, \quad \text{for $i=1,\dots,N$.} \vspace{-2mm} \label{eq:cons_step}
  \end{equation}
  with weights $w_{ij}$ given by an averaging matrix $W$.
  Different consensus steps can possibly use the same averaging matrix.

(iii) \emph{Combinations:}
  Each agent performs the same linear combination of its local variables.

These three operations also allow implicit updates, e.g. updates where the point at which the gradient is evaluated is not explicitly known.
These operations can always be expressed using the vector notation that stacks vertically the agents local variables $x_i \in \Rvec{d}$ and their local (sub)gradients $\nabla f_i (x_i) \in \Rvec{d}$ into vectors $\x,~ \mathbf{g} \in \Rvec{Nd}$ \vspace{-1mm}
\begin{equation} \label{eq:vec_not}
  \x = \begin{bmatrix} x_1 \\\vdots\\ x_N \end{bmatrix}, \qquad \mathbf{g} = \begin{bmatrix} \nabla f_1(x_1) \\\vdots\\ \nabla f_N(x_N) \end{bmatrix}.
\end{equation}
In particular, the consensus step \eqref{eq:cons_step}
can be written for all agents at once as \vspace{-2mm}
\begin{equation}
  \y = (W \otimes I_d) \x, \label{eq:cons} \vspace{-1.5mm}
\end{equation}
where $W \in \Rmat{N}{N}$ is the averaging matrix used in the consensus, which is assumed to be symmetric and generalized doubly stochastic,  $I_d$ denotes the identity matrix of size $d$ and $\otimes$ the Kronecker product. This notation means that we apply the same matrix $W$ for the $d$ dimensions of the agents variables.

\subsection{Generalization of the decentralized problem} \label{sec:gen_prob}
We first define the consensus subspace of $\Rvec{Nd}$.
\begin{definition}[Consensus subspace] \label{def:cons_space}
  The consensus subspace of $\Rvec{Nd}$ is denoted $\C$ and is defined as \vspace{-1mm}
  $$ \C = \left\{ \x \in \Rvec{Nd} ~|~ x_1 = \dots = x_N \in \Rvec{d} \right\}. \vspace{-1mm}$$
  An orthonormal basis of $\C$ is given by $Q_{\paral} = \frac{1}{\sqrt{N}}\qty(\mathbf{1}_N \otimes I_d )$. The dimension of $\C$ is thus $d$.
  The orthogonal complement of $\C$ is denoted $\Cp$ and has dimension $(N-1)d$: \vspace{-1mm}
  $$ \Cp = \left\{ \x \in \Rvec{Nd} ~|~ \x^T\y = 0, \text{ for all } \y \in \C  \right\}.$$
\end{definition} \smallskip
Using this definition, as well as the vector notations, we consider the following optimization problem to generalize the decentralized problem \eqref{opt:dec_prob2} \vspace{-1mm}
\begin{align} \label{opt:gen_prob}
  \underset{\text{\normalsize $\x \in \Rvec{Nd}$}}{\mathrm{min}} &\quad F(\x), \\
  & \hspace{-10mm}\text{such that } \x \in \C, \\[-6mm]
\end{align}
where $F:\Rvec{Nd}\to\Rvec{}$ is a real function from a given class of functions $\mc{F}$.
The class of problems of the form of \eqref{opt:dec_prob2} can be directly related to the class of problems of the form of \eqref{opt:gen_prob} with the additional assumption that $F$ is a separable function.
Let us then define the two classes of functions we are using in this paper, and their corresponding class of separable functions.
The results can be extended to the other classes of functions that are Gram-representable. In the following, we denote by $\Ev$ a finite-dimensional real vector space.
\begin{definition}[$\mc{F}_R$] \label{def:Fr}
  A function $F: \Ev\to\R$ is in $\mc{F}_R(\Ev)$ if $F$ is convex and with (sub)gradient norm bounded as \vspace{0mm}
  $$\|\nabla F(\x) \| \le R \quad \text{ for $R \ge 0$ and for all $\x \in \Ev$}. \vspace{2mm}$$
\end{definition} \smallskip

\begin{definition}[$\mc{F}_{\mu,L}$] \label{def:FmL}
  A function $F: \Ev\to\R$ is in $\mc{F}_{\mu,L}(\Ev)$ if $F$ is $\mu$-strongly convex and $L$-smooth (with $0\le\mu\le L$), i.e. if we have, for all $(\x_1, \x_2) \in \Ev^2$\vspace{-4mm}

  \small
  \begin{align}
    \hspace{-6mm} F(\x_1) \ge F(\x_2) + \big\langle \nabla F(\x_2) \; , \; \x_1-\x_2 \big\rangle &+ \frac{\mu}{2} \|\x_1 - \x_2\|^2, \label{eq:muconv0} \\
    \| \nabla F(\x_1) - \nabla F(\x_1) \| &\le L \|\x_1-\x_2\|.  \label{eq:Lsmooth0}
  \end{align}
  \normalsize
\end{definition}
\medskip

\begin{definition}[$\mc{F}^s$] \label{def:Fs}
  Let $\mc{F}$ be a class of functions, e.g. $\mc{F}_R$ or $\mc{F}_{\mu,L}$. Then $\mc{F}^s(\Ev^N)$ is the set of functions $F:\Ev^N\to\R$ that are separable such that \vspace{-1.5mm}
  $$F(\x) = \frac{1}{N} \sum_{i=1}^N f_i(x_i), \quad \text{with $f_i \in \mc{F}(\Ev)$ for all $i$} \vspace{-1.5mm}$$
  where $\x = (x_1,\dots,x_N) \in \Ev^N$.
\end{definition} \smallskip

We now argue that the worst-case performance of a decentralized algorithm solving problem \eqref{opt:dec_prob2} can be bounded by its performance on problem \eqref{opt:gen_prob}.
The argument relies on the following proposition which follows from Definition \ref{def:Fs} combined with Definitions \ref{def:Fr} and \ref{def:FmL}. \smallskip
\begin{proposition}[Relaxation of separability] \label{prop:relax_gen}
The sets of separable functions (Definition \ref{def:Fs}) are such that
  \begin{align}
    \mc{F}_R^s(\Rvec{Nd}) \subseteq \mc{F}_{\frac{R}{\sqrt{N}}}(\Rvec{Nd}), \text{ and } \mc{F}_{\mu,L}^s(\Rvec{Nd}) \subseteq \mc{F}_{\frac{\mu}{N},\frac{L}{N}}(\Rvec{Nd}). \\[-7.5mm]
  \end{align}
\end{proposition}
Then, any optimization problem (e.g. a PEP) with constraint $F \in \mc{F}_{\frac{R}{\sqrt{N}}}$ (resp. $F \in \mc{F}_{\frac{\mu}{N},\frac{L}{N}}$)
is a relaxation of the one with constraint $F \in \mc{F}_R^s$ (resp. $F \in \mc{F}_{\mu,L}^s$).
We can thus compute worst-case performance bounds of decentralized algorithms for solving \eqref{opt:dec_prob2} by analyzing their performance on the general problem \eqref{opt:gen_prob}. This consists, a priori, in a relaxation since we do no longer exploit the separability of the objective function.

To analyze the performance of an algorithm on problem \eqref{opt:gen_prob} using PEP, we need to express its constraint $\x \in \C$ in a Gram-representable form. To do this, we consider the problem in another coordinates system, which relies on the decomposition of $\Rvec{Nd}$ into subspaces $\C$ and $\Cp$.

\subsection{Change of variables and new coordinates system} \label{sec:change_of_var}
By definition, we have that $ \Rvec{Nd} = \C \oplus \Cp$, so that any vector $\x \in \Rvec{Nd}$ can be uniquely decomposed as
$$ \x = \xvb + \xvp, \qquad \text{with $\xvb \in \C$ and $\xvp \in \Cp$.} $$
Therefore, a vector $\x \in \C$ if and only if $\xvp = 0$. Moreover, as we will see, a consensus step with a generalized doubly stochastic averaging matrix does not impact $\xvb$, but only contracts $\xvp$.
These observations motivate the following change of variable. \smallskip

\begin{definition}[Change of variable $\bx$] \label{def:change_of_var}
\mbox{Let $\Qb = \frac{1}{\sqrt{N}}\qty(\mathbf{1}_N \otimes I_d )$} be the orthonormal basis of $\C$ and $\Qp \in \Rmat{Nd}{(N-1)d}$ the orthonormal basis of $\Cp$. \\
For any $\x \in \Rvec{Nd}$, we define $\bx \in \Rvec{Nd}$ as
\begin{equation} \label{eq:change_of_var}
  \bx = \begin{bmatrix} \xb \\\xp \end{bmatrix} = \frac{1}{\sqrt{N}}\begin{bmatrix} \Qb^T \\[2mm] \Qp^T \end{bmatrix} \x,
\end{equation}
where $\xb = \frac{1}{N} \sum_{i=1}^N x_i \in \Rvec{d}$ and $\xp \in \Rvec{(N-1)d}$. By definition, the matrix $Q = [\Qb~~\Qp]$ is orthogonal, and the change of variable can thus be inverted \emph{via}
\begin{equation} \label{eq:revert_change_of_var}
    \x = \sqrt{N} Q~\bx.
\end{equation}
\end{definition} \smallskip

This change of variable can be interpreted as a change of coordinate system: $\x$ describes the coordinates in the standard basis and $\bx$ its coordinates in a new basis $b$, with $\xb$ the part of the coordinates describing the position along the consensus subspace $\C$, and $\xp$ the part describing the position along the orthogonal complement $\Cp$. %$\sqrt{N}Q$ is thus the basis matrix.
Using Definition \ref{def:change_of_var}, we can easily express the constraint $\x \in \C$.
\begin{proposition} \label{prop:cons_cons}$ \x \in \C \text{ iff } \xp = 0.$
\end{proposition}
\begin{proof}
  We obtain this from \eqref{eq:revert_change_of_var},
  $$ \x = \sqrt{N}\Qb\,\xb + \sqrt{N}\Qp\xp,$$ where the first term is in $\C$ and the second is in $\Cp$.
\end{proof} \smallskip

We now review the impacts of this change of variable on the generalized problem, the formulation of a decentralized algorithm and the Performance Estimation Problem. \smallskip

\subsubsection{Changes in the Generalized Decentralized Problem}
We define a new function $\tF(\bx):\Rvec{Nd}\to\Rvec{}$ depending on the new variable $\bx$ as
\begin{equation}  \label{eq:Ft_def}
  \tF(\bx) = F(\sqrt{N}Q~\bx) = F(\x).
\end{equation}
Using our change of variable and function, as well as Proposition \ref{prop:cons_cons}, problem \eqref{opt:gen_prob} becomes
\begin{align} \label{opt:gen_prob_b}
  \underset{\text{\normalsize $\bx \in \Rvec{Nd}$}}{\mathrm{min}} &\quad \tF(\bx), \\[-0.3mm]
  & \hspace{-10mm}\text{such that } \xp = 0.
\end{align}

Let us analyze what are the properties of the function $\tF$ and its gradient, in view of formulating the Performance Estimation Problem for algorithms solving such problem. \smallskip
\begin{proposition}[Gradients of $\tF$] \label{prop:grad} ~\\
    Let $\nabla \tF(\bx)$ be the gradient of $\tF$ at point $\bx$ and $\nabla F(\x)$ the gradient of $F$ at point $\x$. There holds
    \begin{equation} \label{eq:grad_b}
      \nabla \tF(\bx)= \sqrt{N} Q^T \nabla F(\x).
    \end{equation}
\end{proposition} \smallskip
\begin{proof}
  Using the chain rule, direct computation shows
  $$ \nabla_{\bx} \tF(\bx) = \nabla_{\bx} F(\sqrt{N} Q~\bx) = \sqrt{N} Q^T \nabla_\x F(\x),$$
  where indices for $\nabla$ indicate which variable is considered in the derivatives.
\end{proof}
The gradient $\nabla \tF(\bx)$ can be decomposed in two parts that respectively describe the variation of the value of $\tF(\bx)$ with the consensus and orthogonal components:
\begin{equation} \label{eq:grad_b_decomp}
  \nabla \tF(\bx) = \begin{bmatrix} \nabla_{\paral}~ \tF(\bx) \\ \nabla_\perp \tF(\bx) \end{bmatrix}.
\end{equation}
Proposition \ref{prop:grad} implies the following corollary. \smallskip
\begin{corollary}\label{cor:grad}
  The norm of the gradient of $\tF$ at $\bx$ can be computed as \vspace{-1.5mm}
  $$\|\nabla \tF(\bx)\|_2 = \sqrt{N} \|\nabla F(\x)\|_2. \vspace{1mm} \vspace{-1mm}$$
  Moreover, if $F \in \mc{F}^s$, so that it is separable as $F(\x) = \frac{1}{N}\sum_{i=1}^N f_i(x_i)$, then $\nabla \tF(\bx)$ is the coordinate vector in basis $b$ of the vector \mbox{$\mathbf{g}^T = \qty[\nabla f_1^T(x_1) \dots\nabla f_N^T(x_N)]$}, i.e. \vspace{-1mm}
  \begin{equation} \label{eq:grad_fi}
    \nabla \tF(\bx)= \frac{1}{\sqrt{N}} Q^T \mathbf{g} =\, _b\mathbf{g}. \vspace{-1mm}
  \end{equation}
  In particular, we have that
  $\nabla_{\paral}~ \tF(\bx) = \frac{1}{N}\sum_{i=1}^N \nabla f_i(x_i)$.
\end{corollary}
\medskip
\begin{proposition}[Optimal Conditions of \eqref{opt:gen_prob_b}] \label{prop:opt_cond}
  The vector $\bx^*$ is an optimal solution for problem \eqref{opt:gen_prob_b} if and only if
  $$ \xp^* = 0 \quad \text{ and } \quad \nabla_{\paral} \tF(\bx^*) = 0.$$
\end{proposition}
\begin{proof}
  By definition of problem \eqref{opt:gen_prob_b}, $\xp = 0$ for any feasible solution and thus the problem is equivalent to \vspace{-1mm}
  $$\underset{\text{\normalsize $\bx \in \Rvec{Nd}$}}{\mathrm{min}} \quad \tF\qty(\begin{bmatrix} \xb \\ 0 \end{bmatrix}),  $$
  from which the result follows by convexity.
\end{proof}

\begin{proposition}[Properties of $\tF$] \label{prop:Ft} ~
  In general,
  \begin{enumerate}[(i)]
    \item If $F \in \mc{F}_{\frac{R}{\sqrt{N}}}$, then $\tF \in \mc{F}_{R}$;
    \item If $F \in \mc{F}_{\frac{\mu}{N}, \,\frac{L}{N}}$, then $\tF \in \mc{F}_{\mu,L}$.
  \end{enumerate}
  In cases where $F$ is separable, this becomes
  \begin{enumerate}[(i)]
    \setcounter{enumi}{2}
    \item If $F \in \mc{F}^s_R$, then $\tF \in \mc{F}_{R}$;
    \item If $F \in \mc{F}^s_{\mu,L}$, then $\tF \in \mc{F}_{\mu,L}$.
  \end{enumerate}
\end{proposition}
\begin{proof}
(i) If $F \in \mc{F}_{\frac{R}{\sqrt{N}}}$, then $\|\nabla F(\x)\|_2 \le \frac{R}{\sqrt{N}}$ for all $\x$ and using Corollary \ref{cor:grad}, we have $ \|\nabla \tF(\bx)\|_2 \le R.$
(ii) Starting from Definition \ref{def:FmL} for $F \in \mc{F}_{\frac{\mu}{N}, \,\frac{L}{N}}$, we can use equation \eqref{eq:revert_change_of_var}, \eqref{eq:Ft_def} and \eqref{eq:grad_b}
to express it in term of function $\tF$, its gradients and variables $\bx_1$ and $\bx_2$. Then, using orthogonality of $Q$, we recover conditions for $\tF \in \mc{F}_{\mu,L}$. \\
  Finally, (iii) and (iv) directly follows from (i) and (ii) combined with Proposition \ref{prop:relax_gen}.
\end{proof}

\smallskip
\subsubsection{Changes in the Decentralized Algorithms}~\\
After applying the changes of variables \eqref{eq:revert_change_of_var}, an algorithm for solving problem \eqref{opt:gen_prob} will be suited for solving problem \eqref{opt:gen_prob_b}.
The impact of these changes of variables can be treated independently for each of the three possible operations of a decentralized algorithm from $\M$.

\emph{Gradient:} Using Corollary \ref{cor:grad}, the computation of the vector $\mathbf{g}$ of local gradients of function $F$ at point $\x$ simply becomes the computation of $\nabla \tF(\bx) = \, _b\mathbf{g}$.

\emph{Combinations:} Since $Q$ is orthogonal, the changes of variables does not modify the linear combination operations, e.g. applying \eqref{eq:revert_change_of_var} to the combination $\z = \alpha\x + \beta\y$ gives
$$ \sqrt{N}Q \,_b\z = \sqrt{N}Q(\alpha\, _b\x + \beta\,_b\y) ~\Leftrightarrow~ _b\z = \alpha\, _b\x + \beta\,_b\y. \vspace{-1mm}$$

\emph{Consensus:} Theorem \ref{thm:consb} analyzes the effects of a consensus step on the new variables.
\begin{theorem}[Consensus step in basis $\bn$] \label{thm:consb}
Let $W \in \Rmat{N}{N}$ be a symmetric generalized doubly stochastic matrix.
The consensus step $\y = (W \otimes I_d) \x$, is expressed in basis $b$ as
\begin{equation} \label{eq:cons_basis_b_to_prove}
  \by = \begin{bmatrix} \yb \\ \yp \end{bmatrix} = \begin{bmatrix} \xb \\ \tilde{W} \xp \end{bmatrix},
\end{equation}
where $\tilde{W} \in \Rmat{(N-1)d}{(N-1)d}$ has the same eigenvalues as $W$ with a multiplicity $d$ times larger, except for $\lam_1 = 1$ that is not an eigenvalue of $\Wt$.
\end{theorem}
\begin{proof}
  When applying changes of variables \eqref{eq:revert_change_of_var} in the standard consensus step \eqref{eq:cons}, since $Q$ is orthogonal, we have \vspace{-2.5mm}
  \begin{equation} \label{eq:cons_basis_b}
    \by =\,\bW ~ \bx \quad \text{with } \quad \bW = Q^T(W \otimes I_d)Q, \vspace{-0.5mm}
  \end{equation}
  where $\bW \in \Rmat{Nd}{Nd}$ is the averaging matrix for the consensus expressed in the basis $b$ and has the same eigenvalues as $(W \otimes I_d)$ by similarity.
  By definition, matrix $(W \otimes I_d)$ is symmetric, generalized doubly stochastic, has the same eigenvalues than $W$ with a multiplicity $d$ times larger for each of them, and it can thus be written using the eigen-decomposition: \vspace{-1mm}
  $$W \otimes I_d = V \Lambda V^T = \begin{bmatrix} V_{\paral} & V_\perp \end{bmatrix} \begin{bmatrix} I_d & 0 \\ 0 & \Lambda_\perp \end{bmatrix} \begin{bmatrix}V_{\paral}^T \\[1mm] V_\perp^T \end{bmatrix}, \vspace{-1mm} $$
  where $V_{\paral} = \Qb = \frac{1}{\sqrt{N}}\qty(\mathbf{1}_N \times I_d)$ are the orthonormal eigenvectors corresponding to the $d$ eigenvalues equal to 1 and $V_\perp$ are the orthonormal eigenvectors corresponding to other eigenvalues $\Lambda_\perp$. All the eigenvectors are orthogonal to each other and therefore, since $V_{\paral} = \Qb$, the matrix $\bW$ from \eqref{eq:cons_basis_b} is given by \vspace{-1mm}
    \begin{align}
      \begin{bmatrix} \Qb^T \\[2mm] \Qp^T \end{bmatrix} \begin{bmatrix} V_{\paral} & V_\perp \end{bmatrix} \begin{bmatrix} I_d & 0 \\ 0 & \Lambda_\perp \end{bmatrix}
      \begin{bmatrix}V_{\paral}^T \\[1mm] V_\perp^T \end{bmatrix} \begin{bmatrix} \Qb & \Qp \end{bmatrix}
        &= \begin{bmatrix} I_d & 0 \\ 0 & \Wt \end{bmatrix}, \label{eq:last_bW}\\[-11mm]
    \end{align}
    where $\Wt = \Qp^TV_\perp \Lambda_\perp V_\perp^T\Qp$ and is thus symmetric. Eigenvalues of $\Wt$ are the ones in $\Lambda_\perp$ since $\bW$ and $(W \otimes I_d)$ have the same eigenvalues.
    Finally, we find \eqref{eq:cons_basis_b_to_prove} using this last expression for $\bW$ in the consensus step \eqref{eq:cons_basis_b}.
\end{proof} \smallskip
Theorem \ref{thm:consb} formalizes the separate effects of a consensus step that we discussed at the beginning of Section \ref{sec:change_of_var}.

We now have access to the problem, its properties, and any decentralized algorithm $\mc{M} \in \M$ written using the new variables. There only remains to formulate the PEP that evaluates the worst-case performance of $\mc{M}$ for solving problem \eqref{opt:gen_prob_b} with a given class of functions.
 \smallskip

\subsubsection{Performance Estimation Problem in basis $b$}
This subsection reviews the main ingredients of a PEP \eqref{eq:gen_PEP} after the changes of variables \eqref{eq:revert_change_of_var}, namely, the ingredients of a PEP applied to a decentralized algorithm for solving problem \eqref{opt:gen_prob_b}.
We explain how such PEPs can be formulated as SDPs that are independent of the dimension of vector $\bx \in \Rvec{Nd}$, and thus independent of $N$.

First, the optimal conditions for \eqref{opt:gen_prob_b} are given in Proposition \ref{prop:opt_cond}. We observe that they require the use of specific components $\xp$ and $\nabla_{\paral} \tF$, independently of the full vectors $\bx$ and $\nabla \tF$. This separation of consensus coordinates and orthogonal coordinates also occurs in the effect of consensus step in Theorem \ref{thm:consb}.
To be able to use such components independently from each other, we split the vectors $\bx^k$ and $_b\mathbf{g}^k =\nabla \tF(\bx^k)$ in two parts in the SDP PEP formulation and we consider the scalar product of each part separately.
This results then in two Gram matrices: \vspace{-1mm}
  $$ G_{\paral} = P_{\paral}^TP_{\paral} \quad \text{ and } \quad  G_{\perp} = P_{\perp}^TP_{\perp}, \vspace{-1mm}$$
  where $P_{\paral} \in \Rmat{d}{2K+2}$ and $P_{\perp}\in \Rmat{(N-1)d}{2K+2}$ contains respectively the consensus and the orthogonal components of vectors $\bx^k$ and $_b\mathbf{g}^k$.
  Both Gram matrices have identical sizes, independent of the number of agents $N$.
  They can possibly be combined into a single matrix $G \succeq 0$ as
  $$G = \begin{bmatrix} G_{\paral} & 0 \\ 0 & G_{\perp} \end{bmatrix}.$$

This modification still allows writing all the constraints involving scalar products between the iterates $\bx$ or the gradient $_b\mathbf{g}$, e.g. the usual interpolations conditions for imposing that $\tF \in \mc{F}$, using the fact that \vspace{-1mm}
\begin{align}
  (\bx)^T(\bx) &= \xb^T\xb + \xp^T\xp,\qquad (_b\mathbf{g})^T(_b\mathbf{g}) = \tg_{\paral}^T\tg_{\paral} + \tg_{\perp}^T\tg_{\perp}, \\[-4mm]
  &\text{ and } \quad \bx^T\tg = \xb^T\tg_{\paral} + \xp^T\tg_{\perp}.
\end{align}
The usual interpolation constraints are well linearly Gram-representable as they are expressed as linear combinations of scalar products between $\xb^k, \tg_{\paral}^k$ and between $\xp^k, \tg_{\perp}^k$, and they can thus be embedded in the SDP PEP formulation, see Proposition \ref{prop:GramPEP}.

Concerning the algorithmic steps constraints, we need to decouple them into constraints for the consensus components and for the orthogonal components. Let us analyze separately each possible operation of a decentralized algorithm.

\emph{Gradient and Combinations:} The gradient computation and the linear combinations can directly be decoupled for the consensus and the orthogonal components. These constraints are linear equalities and are thus linearly Gram-representable.

\emph{Consensus:}
For a consensus step with a symmetric and generalized doubly stochastic averaging matrix $W$, the update for the consensus and orthogonal parts is given in Theorem \ref{thm:consb}. For the consensus components, it also corresponds to linear equality that is linearly Gram-representable. For the orthogonal component, the linear equality involves a modified averaging matrix $\Wt$ that cannot be known exactly without distinguishing the different agents individually. Therefore, we consider a spectral description of the symmetric matrix $\Wt$ that specifies its range of eigenvalues: $\lam(\Wt) \in \qty[\lm, \lp]$, with $\lm \le \lp \in \qty(-1,1)$, which is identical to the one of $W$,
excluding $\lam_1(W) = 1$ (Theorem \ref{thm:consb}). We stress that the range $\qty[\lm, \lp]$ is known in our performance analysis, but does not necessarily need to be known for the execution of the algorithms.
We now adapt previous work \cite{PEP_dec} to represent, in a Gram-representable manner, a set of $K$ consensus steps using the same matrix $\Wt$: \vspace{-1mm}
\begin{equation} \label{eq:mat_cons}
  Y = \Wt X \quad \text{with } \quad \Wt \in \Wcl{\lm}{\lp}, \vspace{-1mm}
\end{equation}
where $\Wcl{\lm}{\lp}$ is the class of symmetric matrices of size $(N-1)d$ with eigenvalues between $\lm$ and $\lp$ and $Y,X \in \Rmat{(N-1)d}{K}$ are the matrices whose $k$-th column is the vector involved in the $k$-th consensus step. Consensus steps \eqref{eq:mat_cons} are replaced by necessary LMI Gram-representable constraints that are stated in Theorem \ref{thm:conscons}. These constraints allow the PEP to find worst-case valid over an entire spectral class of averaging matrix. \smallskip
\begin{theorem}[Consensus Constraints from {\cite[Theorem 3]{PEP_dec}}] \label{thm:conscons}
If $Y = \Wt X$ with $\Wt \in \Wcl{\lm}{\lp}$, then the matrix $X^T Y$ is symmetric and
the following LMI Gram-representable constraints are satisfied \vspace{-1mm}
  \begin{align}
    \lm X^T X ~ \preceq  ~X^T Y~ &\preceq ~\lp X^T X, \label{eq:scal_cons} \\
    (Y - \lm X)^T(Y - \lp X) ~&\preceq ~0, \hspace{20mm} \label{eq:var_red} \\[-6.5mm]
  \end{align}
  where the notations $\preceq$ denote negative semi-definiteness.
\end{theorem} \smallskip
The proof of Theorem \ref{thm:conscons} is very similar to \cite[Theorem 3]{PEP_dec} and is omitted.

The last important ingredients of PEP are the initial conditions and the performance criterion. Using $(\xb^k, \xp^k, \tg_{\paral}^k, \tg_{\perp}^k, \tF^k)_{k=0,\dots,K,*}$, as well as all their properties studied in previous subsections, we can formulate any metric that is not agent-specific. Propositions \ref{prop:init2} and \ref{prop:init1} gives two examples for initial conditions.
Metrics used for initial conditions can also be used as performance criteria. \smallskip
\begin{proposition} \label{prop:init2}
  The initialization of variables $x_i^0 \in \Rvec{d}$, for $i=1,\dots,N$ such that \vspace{-1mm}
  \begin{equation} \label{eq:init2}
    x_i^0 = x^0 \text{ for all $i$, and } \quad \|x^0 - x^*\|_2^2 \le D^2, \vspace{-1mm}
  \end{equation}
  where $x^* \in \Rvec{d}$ is an optimal solution of \eqref{opt:dec_prob}, is equivalent to the following linearly Gram-representable constraints \vspace{-1mm}
  $$ \xp^0 = 0, \quad \text{ and } \quad \|\xb^0 - \xb^*\|_2^2 \le D^2, \vspace{-1mm}$$
  where $_b \x^* = \begin{bmatrix} \xb^* \\ 0 \end{bmatrix} \in \Rvec{Nd}$ is an optimal solution of \eqref{opt:gen_prob_b}.
\end{proposition}
\begin{proof}
  All the agents start with the same $x_i^0$, which means that $\x^0 \in \C$ and thus $\xp^0=0$, see Proposition \ref{prop:cons_cons}. We also have, by definition, that $\xb^0 = x^0$ and $\xb^* = x^*$. These conditions are linearly Gram-representable since they involve linear combinations of scalar products of variables.
\end{proof} \medskip

\begin{proposition} \label{prop:init1}
  The initialization of variables $x_i^0 \in \Rvec{d}$, for $i=1,\dots,N$ such that \vspace{-1mm}
  \begin{equation} \label{eq:init1}
    \frac{1}{N} \sum_{i=1}^N \|x_i^0 - x^*\|_2^2 \le D^2, \vspace{-1mm}
  \end{equation}
  where $x^* \in \Rvec{d}$ is an optimal solution of \eqref{opt:dec_prob}, is equivalent to the following linearly Gram-representable constraints
  \begin{equation} \label{eq:init1_b}
    \|\bx^0 -\, \bx^*\|_2^2 = \|\xb^0 - \xb^*\|_2^2 + \|\xp^0 - \xp^*\|_2^2 \le D^2,
\end{equation}
  where $\bx^* \in \Rvec{Nd}$ is an optimal solution of \eqref{opt:gen_prob_b}.
\end{proposition}
\begin{proof}
  Using change of variables \eqref{eq:revert_change_of_var}, the initial condition \eqref{eq:init1} becomes \vspace{-2mm}
  \begin{equation}
    \frac{1}{N} \sum_{i=1}^N \|x_i^0 - x^*\|_2^2 = \frac{1}{N} \|\x^0 -\x^*\|_2^2 = \|\bx^0 -\, \bx^*\|_2^2 \le D^2. \vspace{-1mm}
  \end{equation}
  We then obtain \eqref{eq:init1_b} by definition of the Euclidean norm. This condition is linearly Gram-representable since it involves a linear combination of scalar products of variables.
\end{proof}
This last example can also apply for metric $\sum_{i=1}^N \|x_i^0 - x^*\|_2^2$, but in that case, $N$ would appear as a scaling parameter for the formulation in basis $b$.

We now have all the ingredients to build a PEP formulation that can evaluate bounds on the worst-case performance of any decentralized algorithm under consideration (see Section \ref{sec:algo}) to solve problem \eqref{opt:dec_prob2} for different classes of functions ($\mc{F}_R$ and $\mc{F}_{\mu,L}$), and whose size is independent of the number of agents $N$. The bounds are valid for any averaging matrices that is symmetric, generalized doubly-stochastic and with the given range of eigenvalues.
The obtained bounds may scale with $N$ as a parameter but we can choose the initial conditions and the performance criterion to obtain bounds totally independent of $N$.

\section{Numerical results} \label{sec:numres}
\subsection{Comparison with the previous PEP formulation \cite{PEP_dec}}

\begin{figure*}[h]
  \centering \hspace{-3mm}
    \begin{subfigure}{0.5\textwidth}
        \includegraphics[width=\textwidth]{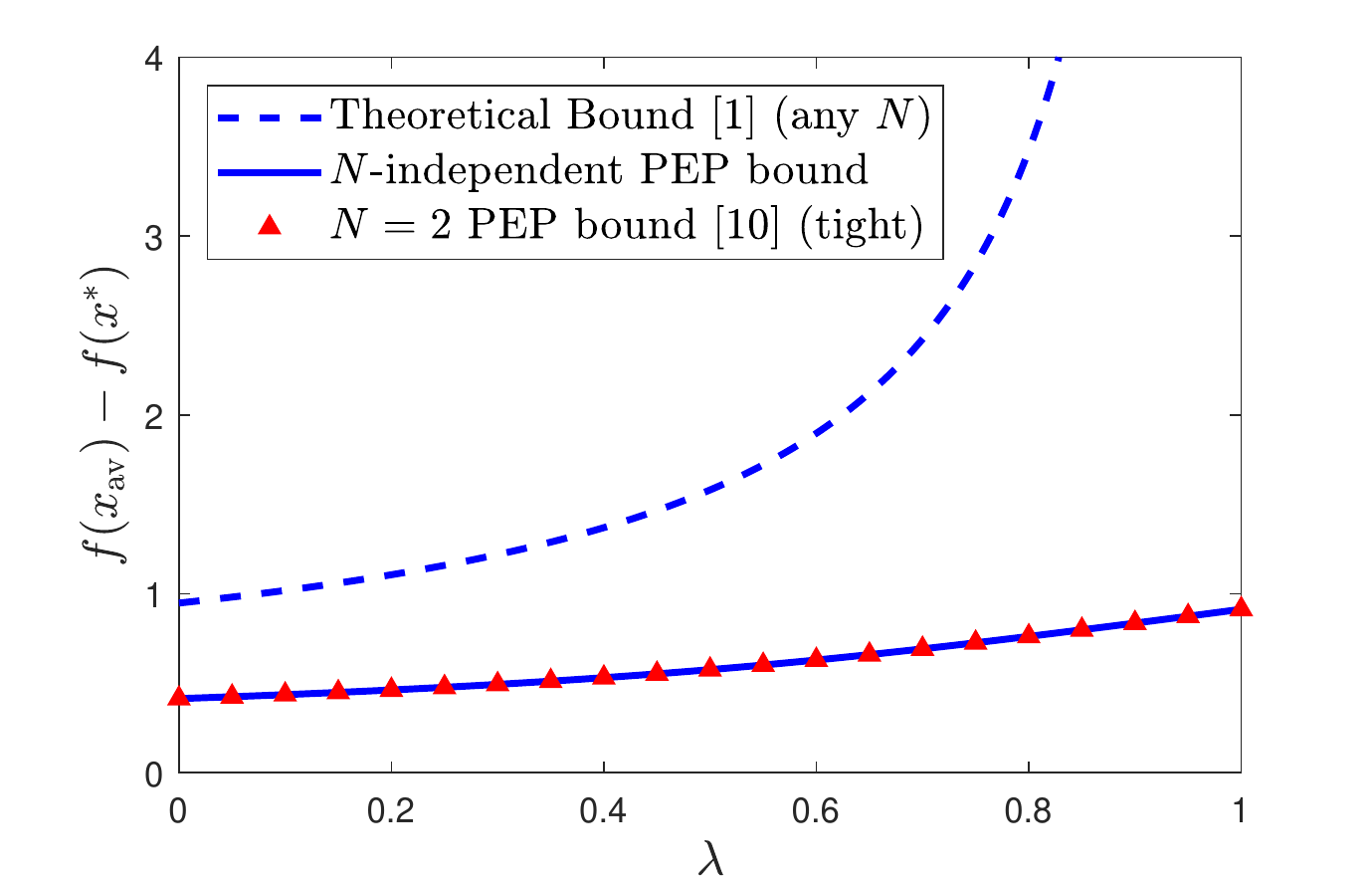}
        \caption{DGD}
        \label{fig:DGD}
    \end{subfigure} \hspace{-3mm}
    \begin{subfigure}{0.5\textwidth}
      \includegraphics[width=\textwidth]{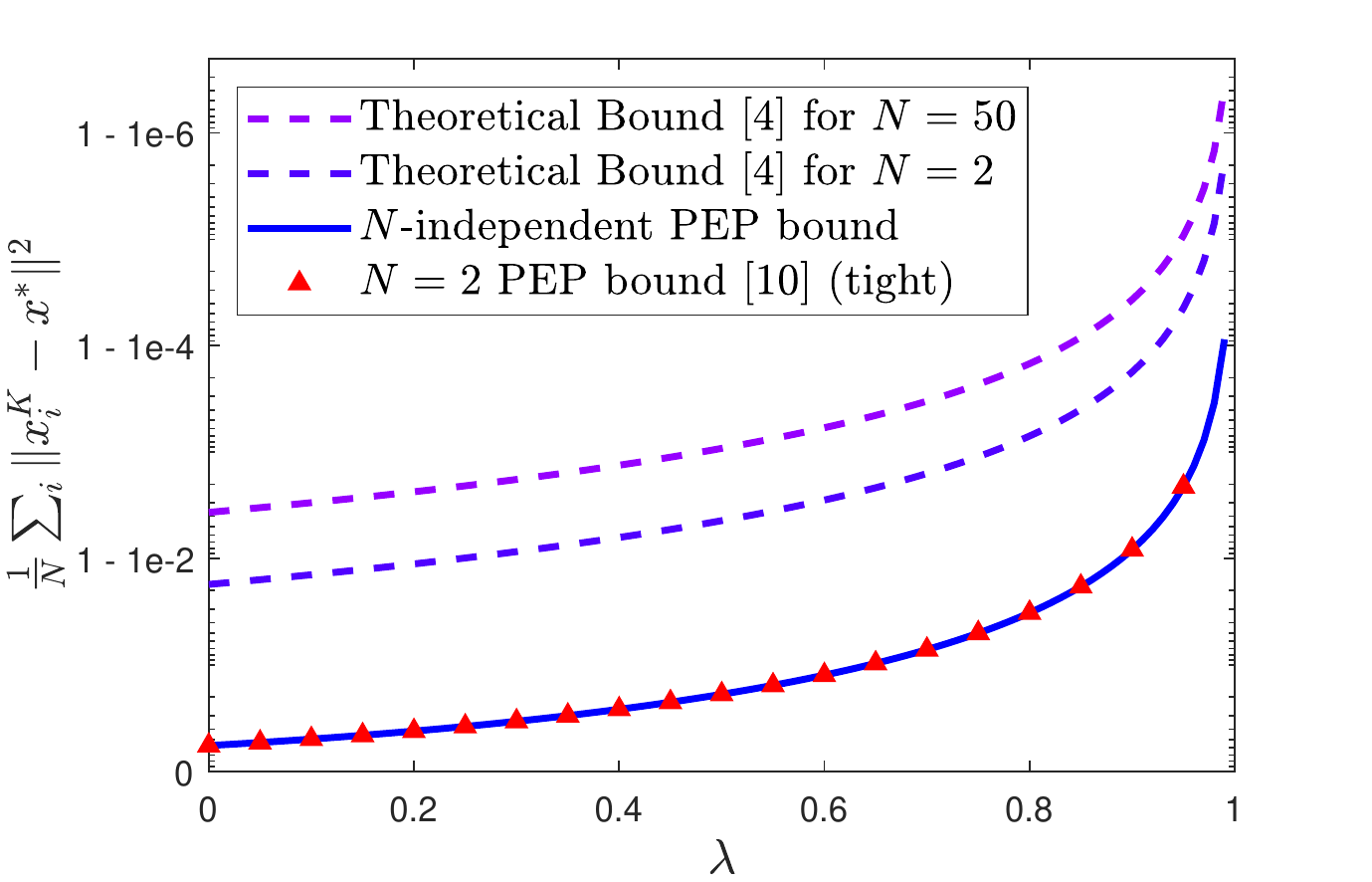}
        \caption{DIGing}
        \label{fig:DIGing}
    \end{subfigure}
    \caption{Evolution with $\lambda$ of the worst-case performance of
    $K = 10$ iterations of DGD and DIGing. The agent-independent PEP bounds from our new formulation exactly match the PEP bounds for $N=2$ from previous PEP formulation \cite{PEP_dec} and improve on the theoretical guarantees.
    (a) DGD with $f_i \in \mc{F}_R$ for all $i$ ($R=1$) and $\alpha = \frac{1}{\sqrt{K}}$; (b) DIGing with $f_i \in \mc{F}_{\mu,L}$ for all $i$ ($\mu=0.1$, $L=1$) and optimized $\alpha$. \vspace{-5mm}}
  \label{fig:compDGD_DIGing}
\end{figure*}

We implement this $N$-independent PEP formulation for analyzing the performance of $K$ iterations of DGD and DIGing with constant step-size in the same settings as in \cite[Section IV]{PEP_dec}, to compare the results with both the previous $N$-\,dependent PEP formulation \cite{PEP_dec} and the theoretical guarantees (Fig. \ref{fig:compDGD_DIGing}).
We consider a class of averaging matrices that contains symmetric generalized doubly stochastic matrices, with a range of eigenvalues $\qty[-\lam,\lam]$ for $\lam \in [0,1)$ that holds for all eigenvalues, except for $\lam_1 = 1$.
Our framework allows for any performance criterion to be chosen, but we use the same as the theoretical guarantees we are comparing with.

For DGD, we analyze the worst-case of $f(\xmoy) - f(x^*)$, where $f(x) = \frac{1}{N}\sum_{i=1}^N f_i(x)$ is the function we minimize in the decentralized problem \eqref{opt:dec_prob}, $x^*$ is an optimum of $f$, and $\xmoy = \frac{1}{N(K+1)} \sum_{i=1}^N \sum_{k=0}^K x_i^k $ is the average iterates over all the iterations and all the agents.
This performance criterion has a theoretical guarantee \cite{DGD} that is shown in Fig. \ref{fig:DGD} for comparison. This guarantee applies to a step-size $\alpha = \frac{1}{\sqrt{K}}$ and local functions $f_i \in \mc{F}_R$ for all $i$, i.e. convex functions with subgradient norm bounded by $R$. The initial conditions are the ones described in Proposition \ref{prop:init2}. In our experiments, we set $R = 1$ and $D = 1$, but results for general values of $R$ and $D$ can be recovered by scaling.

Fig. \ref{fig:DGD} shows the evolution as a function of $\lam$ of the results obtained for $10$ iterations. We observe that the spectral bound obtained with our new agent-independent PEP formulation matches those obtained with the previous agent-dependent PEP formulation for different values of $N$ \cite{PEP_dec}. The independence of $N$ in this spectral bounds from \cite{PEP_dec} was observed but not guaranteed. However, the new bound is known to be valid for any value of $N$. This means that
\begin{enumerate}[(i)]
  \item The relaxation of the separability of the objective function that we have introduced with the generalized decentralized problem (Section \ref{sec:gen_prob}) has no impact on the PEP results for DGD in such settings.
  \item The agent-independent PEP formulation is numerically tight for DGD with symmetric generalized doubly stochastic averaging matrices since this is the case for the agent-dependent formulation \cite{PEP_dec}.
\end{enumerate}

\noindent The same observations also apply to DIGing whose results are shown in Fig. \ref{fig:DIGing}, for a different setting. We consider metric from Proposition \ref{prop:init1} as both initial condition and performance criterion, we suppose $f_i \in
\mc{F}_{\mu,L}$ for all $i$, i.e. local functions are $\mu$-strongly convex and $L$-smooth.
In our experiments, we set $D = 1$, $\mu = 0.1$ and $L=1$ but results for general values of $D$ and $L$ can be recovered by scaling. Other values of $\mu$ lead to similar observations.
Resulting PEP and theoretical bounds depend on the step-size $\alpha$. For each bound and each $\lam$, we separately choose the step-size leading to the smallest value.
For the spectral PEP formulations, the best values of $\alpha$ for $\mu=0.1$ are found empirically: $\alpha = 0.44(1-\lam)^2$.
Fig. \ref{fig:DIGing} uses a special log-scale for the y-axis, shifted by $1$, the initial value of the performance criterion.
We see that the theoretical guarantee from \cite{DIGing} for $N=2$ remains close to 1 for all values of $\lam$ and this gets worst for larger $N$. This is not the case for our spectral agent-independent PEP bound which improves on the theoretical one.

\subsection{Comparison of DIGing and EXTRA}
Our framework enables easy comparison of algorithms based on their worst-case performance bounds. These upper bounds, and thus the resulting comparisons are valid over the entire given spectral class of networks, for all network sizes.
For example, Fig. \ref{fig:algo_comp} compares DIGing \cite{DIGing} and EXTRA \cite{EXTRA} performance after 10 iterations, in the same setting as in Fig. \ref{fig:DIGing}: we consider metric from Proposition \ref{prop:init1} both as initial condition and performance criterion and $f_i \in \mc{F}_{\mu,L}$ for all $i$.
In our experiments, we set $D = 1$, $\mu = 0.1$ and $L=1$ but results for general values of $D$ and $L$ can be recovered by scaling. We also consider the class of symmetric generalized doubly stochastic averaging matrices with range of eigenvalues $\qty[-\lam,\lam]$.
The algorithm EXTRA \cite{EXTRA} uses a second averaging matrix $W_2$ which the authors recommend to set to $W_2 = (I+W)/2$. Consensus steps using $W_2$ are thus considered as $W_2 x = \frac{x}{2} + \frac{Wx}{2} $.\\
We compare the two algorithms in two different cases: \vspace{-0.5mm}
\begin{itemize}
  \item \textbf{Constant averaging matrices}. It means that the same matrix is used for each consensus. In that case, the conditions from Theorem \ref{thm:conscons} are applied in PEP to all consensus steps at once.
  \item \textbf{Time-varying averaging matrices}. It means that each consensus step can potentially uses a different matrix from the same spectral class. In that case, the conditions from Theorem \ref{thm:conscons} are applied separately to each consensus steps in PEP. Therefore, for each of these applications of Theorem \ref{thm:conscons}, $X$ and $Y$ have only one column and LMIs become scalar inequalities.
\end{itemize}
For each case, the step-size $\alpha$ is optimized for each value of $\lam$. In Fig. \ref{fig:algo_comp}, we observe that, for constant averaging matrices (plain lines),
the worst-case performance of EXTRA after 10 iterations is better than the one DIGing for all the values of $\lam$. However, EXTRA seems to be much less robust to time-varying matrices (dashed line). Indeed, in that case, DIGing has exactly the same performance as for constant matrices, while the performance of EXTRA deteriorates: it remains better than DIGing for small values of $\lam$ but explodes when $\lam$ increases. The robustness of DIGing to time-varying matrices was already known since its theoretical analysis \cite{DIGing}, despite its non-tightness (see Fig. \ref{fig:DIGing}), is valid for time-varying matrices. However, the worst-case behavior of EXTRA for time-varying matrices was unknown since its theoretical analysis \cite{EXTRA} only focuses on constant matrices.

\begin{figure}[b]
  \vspace{-6mm}
  \centering
  \includegraphics[width=0.5\textwidth]{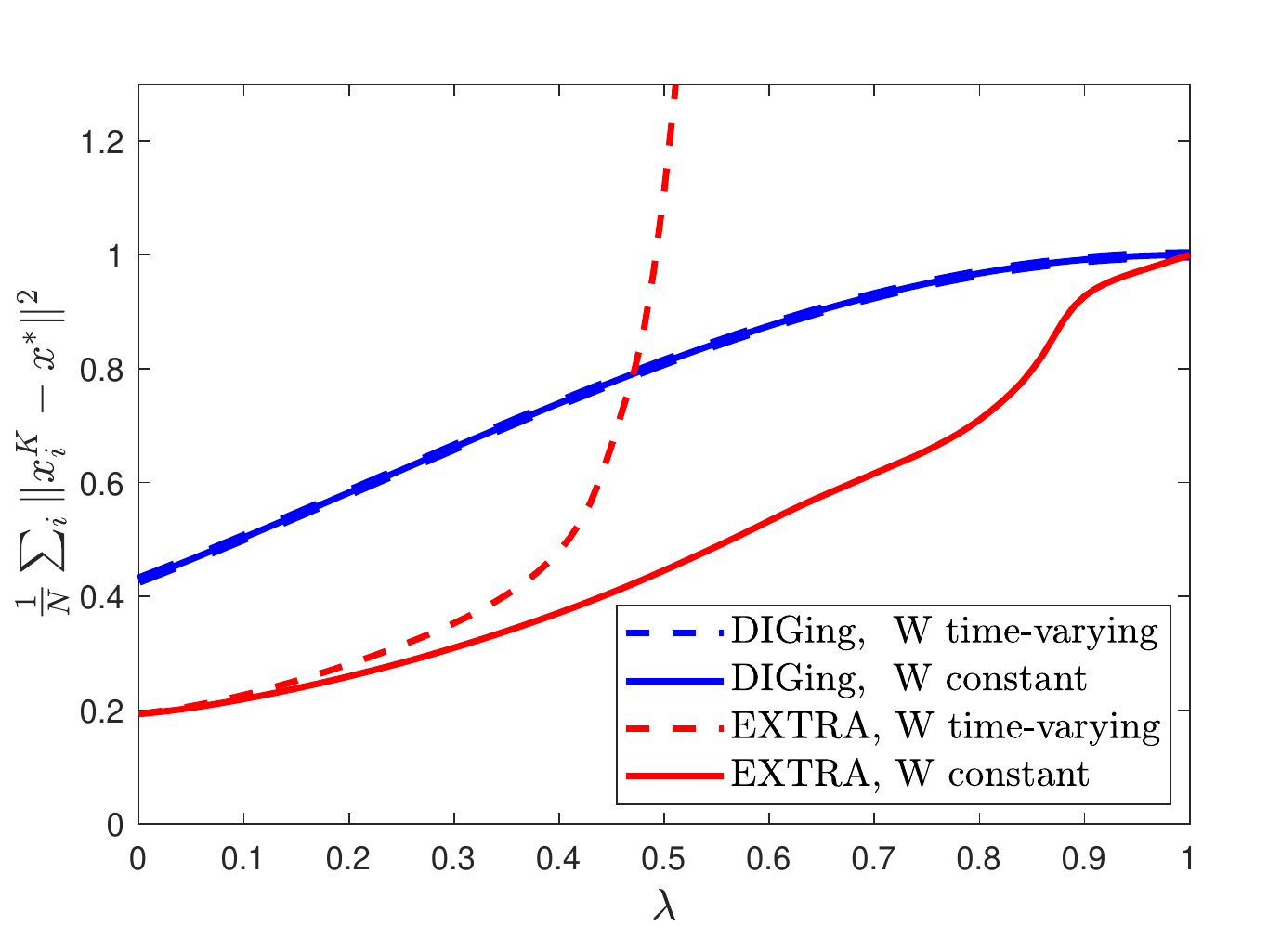}
  \caption{Comparison between DIGing and EXTRA for $10$ iterations with constant and time-varying averaging matrices. Evolution is shown with the range of eigenvalues $\qty[-\lam,\lam]$. Optimized $\alpha$ and $f_i \in \mc{F}_{\mu,L}$ for all $i$ ($\mu=0.1$, $L=1$).}
  \label{fig:algo_comp}
\end{figure}

Note that these observations hold after 10 iterations but similar ones were made for larger number of iterations. However, the value of $\lambda$ where dashed lines cross may slightly vary.
Our observations are consistent with the fact that the alternative IQC methodology \cite{IQC_dec} does not provide any convergence guarantee for EXTRA with similar settings when $\lam \ge 0.6$. Indeed, IQC evaluates the performance over one iteration of the algorithm and can therefore only consider the case where the averaging matrices are time-varying.

\section{Conclusion}
We develop a methodology to automatically compute numerical worst-case bounds of a broad class of decentralized optimization methods that are independent of the number of agents in the problem. This methodology also allows considering time-dependent settings. This opens the way for computer-aided analysis of many decentralized algorithms, which could lead to improvements in their performance guarantees and parameter tuning.
Moreover, the guarantees computed with our tool appear to be tight in many cases.

This methodology is based on the approach of performance estimation problem (PEP) and reuses results from our previous work to represent spectral classes of matrices. Although it formulates a relaxed problem, it appears to provide tight worst-case bounds for DGD and DIGing, which are valid for any number $N$ of agents and improve on the theoretical guarantees.

\bibliographystyle{IEEEtran}
\bibliography{refs.bib}
%\printbibliography

\end{document}